\documentclass[a4paper]{amsart}

\newcommand{\bH}{\mathbb{H}}

\newcommand{\bN}{\mathbb{N}}

\newcommand{\bR}{\mathbb{R}}

\newcommand{\bZ}{\mathbb{Z}}

\usepackage{url,graphicx,verbatim,hyperref}
\usepackage[all]{xy}
\usepackage[hmargin=3cm,vmargin=3.5cm]{geometry}

\newtheorem{theorem}{Theorem}[section]
\newtheorem{lemma}[theorem]{Lemma}
\newtheorem{proposition}[theorem]{Proposition}

\theoremstyle{definition}

\newtheorem{remark}[theorem]{Remark}

\newcommand{\mr}[1]{{\rm #1}}

\title{An elementary proof of the string topology structure of compact oriented surfaces}
\author{A.P.M. Kupers}
\date{May 27, 2011}

\begin{document}

\begin{abstract}We give a new proof of the string topology structure of a compact-oriented surface of genus $g \geq 2$, using elementary algebraic topology. This reproves the result of Vaintrob.\end{abstract}

\maketitle

\section*{Introduction}

Let $\Sigma_g$ be a compact oriented surface of genus $g$, then $\Sigma_g$ is an Eilenberg-Mac Lane space $K(\pi_g,1)$ for $\pi_g$ the surface group $\pi_1(\Sigma_g,p)$, where $p \in \Sigma_g$ is any basepoint. This group has a presentation given by:
\[\pi_g = \langle A_1,\ldots,A_g,B_1,\ldots,B_{g} \rangle/([A_1,B_1]\cdots[B_g,A_g])\]

We will calculate the string topology structure on $H_*(L\Sigma_g;\bZ)$ for $g \geq 2$, i.e. we will give the string product, BV-operator and string coproduct as defined in \cite{chassullivan} \cite{cohenvoronov}. These comprise all explicitly known non-zero string operations. In doing this we reprove the results of Vaintrob \cite{vaintrob}. The cases $g = 0$ and $g=1$ are a consequence of Menichi's work on the string topology of spheres \cite{menichi} and, in the second case, of the fact that the string topology structure of a product is the product of the string topology structures. From now we always assume $g \geq 2$.

\section{A property of $\pi_g$}
We will need the following fact: the centralizer of $h \in \pi_g$ is infinite cyclic unless $h = e$, in which case it is obviously the entire group $\pi_g$.

\begin{proposition}\label{propcentralizer} If $h \neq e \in \pi_g$, then the centralizer of $h$ is infinite cyclic.\end{proposition}

\begin{proof}We will use the theory of Riemann surfaces and in particular of $PSL(2,\bR)$ as isometries of the upper half-plane as model for hyperbolic space. We also use that all discrete subgroups of $\bR$ are isomorphic to $\bZ$.

A particular universal cover of $\Sigma_g$ is the upper-half plane $\bH$ and $\pi_g$ can be made to act -- of course still freely and properly discontinuous -- on $\bH$ by isometries. This means that we have an inclusion $\pi_g \hookrightarrow PSL(2,\bR)$. Because the fundamental domain has non-zero volume, the image has to be discrete. The non-identity elements of $PSL(2,\bR)$ can be classified into three types, depending on their fixed points when acting on the compactification $\bar{\bH}$. 
\begin{enumerate}
	\item Elliptic elements have a fixed point in $\bH$, so can't be in the image of $\pi_g$. 
	\item Parabolic elements have a single fixed point on the boundary and are all conjugate to $z \mapsto z + 1$. Thus, if $h$ is sent to a parabolic element, without loss of generality we can assume it is sent to $z \mapsto z+1$. Another element in $PSL(2,\bR)$ only commutes with this if it is of the form $z \mapsto z+t$. This means that the image of the centralizer of $h$ must lie in $\{z \mapsto z+t| t \in \bR\} \cong \bR$ and be discrete, hence be infinite cyclic. 
	\item Hyperbolic elements have two distinct fixed points on the boundary and are all conjugate to $z \mapsto \lambda z$ for some $\lambda \in \bR_{>0}$. An element of $PSL(2,\bR)$ only commutes with this if it is of the form $z \mapsto \rho z$ for some $\rho \in \bR_{>0}$. Hence, if $h$ is sent to a hyperbolic element without loss of generality it is sent to $z \mapsto \lambda z$. Then the centralizer of $h$ is sent to a discrete subgroup of $\{z \mapsto \rho z| \rho \in \bR_{>0}\} \cong \bR$ and again we conclude that it is infinite cyclic.
\end{enumerate}
\end{proof}

\begin{remark}In fact, because we required $g \geq 2$, the element $h$ always has to be mapped to an hyperbolic element.\end{remark}

\section{The homotopy and integral homology groups of $L\Sigma_g$}
We will show that that the fact that $\Sigma_g$ is a $K(\pi_g,1)$ implies that $L\Sigma_g$ is a disjoint union of $\Sigma_g$ with a number of circles, which makes its homology easy to compute. First note that $\pi_0(L\Sigma_g)$ consists of the conjugacy classes in $\pi_g$, which in canonical bijection with the set $L$ of isotopy classes of closed curves on $\Sigma_g$. This means we get a connected component $L_{[h]}\Sigma_g$ for each conjugacy class of $\pi_g$ or equivalently for each isotopy class of closed curve. 

Consider the long exact sequence of homotopy groups associated to the restriction of the fibration $\Omega \Sigma_g \to L \Sigma_g \to \Sigma_g$ to each of the connected components $L_{[h]}\Sigma_g$ of $LG$. Because the connected components of $\Omega \Sigma_g$ are contractible and $\Sigma_g$ is a $K(\pi_g,1)$, we conclude that $L_{[h]}\Sigma_g$ is a $K(G,1)$ as well. But for which groups $G$? To find out, we fix a basepoint $p \in \Sigma$ and a based representative $h$ of $[h]$ and use the following lemma:

\begin{lemma}\label{lemcentralizerpi1} $\pi_1(L_{[h]}\Sigma_g,h)$ is given by the centralizer in $\pi_1(\Sigma_g,p)$ of the representative $h$ of the conjugacy class $[h]$.\end{lemma}

\begin{proof}Although this fits into the more general context of Whitehead products on homotopy groups, we give an elementary proof. An element of this group is a homotopy class of maps $S^1 \to L\Sigma_g$ which sends $1$ to the based loop $h$. This is the same as a homotopy class of maps $S^1 \times S^1 \to \Sigma_g$ which sends $S^1 \times 1$ to $h$. Consider the other restriction $f: 1 \times S^1 \to \Sigma_g$ and hence the map $(h,f): S^1 \vee S^1 \to \Sigma_g$. The fact that this factors over $S^1 \times S^1$ is equivalent to the fact that $hfh^{-1}f^{-1}$ is homotopically trivial, which is in turn equivalent to $f$ being a centralizer of $h$. Conversely, any element of the centralizer gives such a map.\end{proof}  

Let $e$ denote the constant loop at $p$. Using proposition \ref{propcentralizer} we obtain that $\pi_1(L_{[e]}\Sigma_g,e) = \pi_g$ and $\pi_1(L_{[h]}\Sigma_g,h) \cong \bZ$ for $[h] \neq [e]$. The former implies $L_{[e]} \Sigma \simeq \Sigma$, and indeed evaluation at the basepoint and inclusion of constant loops are inverse up to homotopy. The latter tells us that $L_{[h]}\Sigma$ is homotopy equivalent to a circle for $[h] \neq [e]$. This allows for a complete calculation of the homology of $L\Sigma_g$.

\begin{theorem}The integral homology of $L\Sigma_g$ is given by

\[H_*(L\Sigma_g;\bZ) = \bigoplus_{[h] \in \mr{Conj}(\pi_g)} H_*(L_{[h]}\Sigma_g;\bZ) \cong \begin{cases} \bZ[L] & \text{if $*=0$} \\
H_1(\Sigma_g;\bZ) \oplus \bZ[L]/(\bZ \cdot [e]) & \text{if $* =1$} \\
\bZ & \text{if $*=2$} \\
0 & \text{otherwise} \end{cases}\]\end{theorem}

Here we have to be careful: in $H_1$, the summand corresponding to a conjugacy class $[h]$ is \textit{not} always generated by the class that one would expect.

Let $k_h$ be a loop generating the centralizer of $h$. Note that we can arrange that $h = k_h^l$ for some $l \in \bN$, since $h$ lies in its own centralizer. The example of $A_1^2$ shows that we can have $l \neq 1$. It is easy to see that $l$ is independent of the choice of representative $h$. We call $l = l([h])$ the level of the conjugacy class.\footnote{It is essentially a winding number: $l$ is the largest number such that there exists a representative $S^1 \to \Sigma_g$ of $h$ which factors as $S^1 \overset{z \mapsto z^l}{\to} S^1 \to \Sigma_g$. If $l=1$ then the corresponding isotopy class of closed curve is called primitive, exactly because it can not be written as a power of a smaller closed curve.} Using the proof of the lemma \ref{lemcentralizerpi1} we now see that a generator of $H_1(L_{[h]}\Sigma_g;\bZ)$ is given by the homology class of the cycle 
\[\tilde{h} := \theta \mapsto \left(t \mapsto h\left(t+\frac{\theta}{l}\right)\right)\]

\section{The BV-operator} Rotation of loops gives an action of $S^1$ on the free loop space: 
\begin{align*}\rho: S^1 \times L\Sigma_g &\to L\Sigma_g \\
(\theta,\gamma) &\mapsto (t \mapsto \gamma(t+\theta))\end{align*}

The BV-operator is given by $\Delta(a) = \rho_*([S^1] \times a)$ and we can calculate it explicitly. Note that $\rho$ respects connected components, so it suffices to calculate the BV-operator for each connected component separately. We start with $[h] = [e]$. Then the action $\rho$ is homotopy trivial, so $\Delta$ vanishes on the homology of $L_{[e]}\Sigma_g$. 

Now let's do $[h] \neq [e]$. Only in degree one could the BV-operator be non-zero. If we pick a representative $h$ as our generator for $H_0(L_{[h]}\Sigma_g;\bZ)$ we see directly that $\Delta(h) = l([h])\tilde{h}$, where $l$ is the level of the conjugacy class of $h$ as before.

\section{The string product and coproduct} We will use the actual definition of the string product as being induced by the following diagram
\[\xymatrix{L\Sigma_g \times L\Sigma_g \ar[d]_{\mr{ev} \times \mr{ev}} & \ar[l]_(.55){i} \mr{Map}(S^1 \vee S^1,\Sigma_g) \ar[d]^{\mr{ev}} \ar[r]^(.65){j} & L\Sigma_g \\
\Sigma_g \times \Sigma_g & \ar[l]^\nabla \Sigma_g & }\]

\noindent as $H_*(L\Sigma_g;\bZ) \otimes H_*(L\Sigma_g;\bZ) \ni a \otimes b \mapsto j_* i^!(a \otimes b) \in H_{*-d}(L\Sigma_g;\bZ)$. Because our surfaces are even-dimensional we do not need to worry about signs.

We know that the string product drops the degree by -2 and has a unit in $H_2(L\Sigma_g;\bZ)$ given by the image $c_*([\Sigma_g])$ of fundamental class of $\Sigma_g$ under the induced map $c_*: H_*(\Sigma_g;\bZ) \to H_*(L_{[e]} \Sigma_g;\bZ)$ of the constant map. This unit is therefore exactly the generator of $H_2(L\Sigma_g;\bZ)$. Hence it suffices to calculate the string product on $H_1(L\Sigma_g;\bZ) \otimes H_1(L\Sigma_g;\bZ)$, which is mapped to $H_0(L\Sigma_g;\bZ)$.

The connected components of $L\Sigma_g \times L\Sigma_g$ are given by the products $L_{[h_1]}\Sigma_g \times L_{[h_2]}\Sigma_g$. If $[h_1] = [e] = [h_2]$ then the vertical arrows become homotopy equivalences and the string product reduces to the ordinary intersection product $\langle-,-\rangle$ on the homology of $\Sigma_g$: $A_i \cdot A_j = B_i \cdot B_j = 0$ and $A_i \cdot B_j = \delta_{ij}$.

With two lemma's we calculate the other two cases. The first case is $[h_1] = [e]$ and $[h_2] = [h] \neq [e]$.

\begin{lemma}Let $\tilde{h}\in H_1(L_{[h]}\Sigma_g,\bZ)$ be the generator. We have that $A_i \cdot \tilde{h} = \frac{1}{l([h])}\langle A_i, h \rangle h$, where $\langle-,-\rangle$ is the ordinary intersection product. A similar formula holds for $B_i \cdot \tilde{h}$.\end{lemma}

\begin{proof}In this case the diagram reduces to
\[\xymatrix{\Sigma_g \times L_{[h]} \Sigma_g \ar[d]_{id \times \mr{ev}} & \ar[l]_(.4){\iota = \mr{ev} \times id} L_{[h]} \Sigma_g \ar[d]^{\mr{ev}} \ar[r]^{id} & L_{[h]}\Sigma_g \\
\Sigma_g \times \Sigma_g & \ar[l]^\nabla \Sigma_g & }\]

The class $i_!(A_i \otimes \tilde{h})$ must be a multiple of the generator $h$ of $H_0(L_{[h]}\Sigma_g;\bZ)$. It is given by $(id \times \mr{ev})^*(u) \cap (A_i \otimes \tilde{h})$, where $u$ is Thom class for $\nabla$. To determine which multiple we can compose with $\mr{ev}_*$. We get the following formula:
\[\mr{ev}_*((id \times \mr{ev})^*(u) \cap (A_i \otimes \tilde{h})) = u \cap (A_i \otimes \mr{ev}_*(\tilde{h})) = u \cap (A_i \otimes k_h)\]

\noindent where $k_h$ was the generator of the centralizer of $h$. This is the intersection product $\langle A_i,k_h \rangle$, which in turn is equal to $\frac{1}{l([h])}\langle A_i, h \rangle$ because by definition $h = (k_h)^{l([h])}$. This proves the formula.
\end{proof}

The second case is $[h_1], [h_2] \neq [e]$. The result uses the Goldman bracket \cite{goldman}.

\begin{lemma}For $i=1,2$, let $\tilde{h}_i\in H_1(L_{[h_i]}\Sigma_g,\bZ)$ be the generator. We have that $\tilde{h}_1 \cdot \tilde{h}_2$ is equal to a multiple of the Goldman bracket: 
\[\tilde{h}_1 \cdot \tilde{h}_2 = \frac{1}{l([h_1])l([h_2])}[h_1,h_2]\]\end{lemma}

\begin{proof}It is well-known that the Goldman bracket $[h_1,h_2]$ of $h_1$ and $h_2$ is given by $\Delta(h_1) \cdot \Delta(h_2)$ \cite[example 7.1]{chassullivan}. We just have to note that $\Delta(h_1) = l([h_1])\tilde{h}_1$, $\Delta(h_2) = l([h_2])\tilde{h}_2$ and that the string product and the Goldman bracket are bilinear. Because all homology groups are free abelian, we do not need to worry about torsion.\end{proof}

Finally, we get to the string coproduct. The string coproduct is almost trivial: it vanishes except on $H_2(L_{[e]}\Sigma_g;\bZ)$ and is given on the unit as follows:
\[H_2(L_{[e]}\Sigma_g;\bZ) \ni c_*([\Sigma_g]) \mapsto \chi(\Sigma_g)[e] \otimes [e]= (2-2g)[e] \otimes [e] \in H_0(L_{[e]}\Sigma_g;\bZ) \otimes H_0(L_{[e]}\Sigma_g;\bZ)\]

This is a consequence of results by Tamanoi \cite{tamanoicoprod}.

\bibliographystyle{amsalpha}
\bibliography{simplesurfaces}

\end{document}